\documentclass[12pt, a4paper, oneside]{article}
\usepackage{amsmath, amsthm, amssymb, graphicx}
\usepackage[bookmarks=true, colorlinks, citecolor=blue, linkcolor=black]{hyperref}
\usepackage{appendix}
\newtheorem{theorem}{Theorem}[section]
\newtheorem{lemma}[theorem]{Lemma}
\newtheorem{proposition}[theorem]{Proposition}

\newtheorem{definition}{Definition}[section]
\newtheorem{remark}{Remark}[section]

\title{The existence and uniqueness of infinite combinatorial Yamabe flows}
\author{Bohao Ji}
\date{}

\begin{document}

\maketitle

\begin{abstract}
    In this paper, we study the combinatorial Yamabe flow on infinite triangulated surfaces in Euclidean background geometry, aiming for solving discrete Yamabe problem on noncompact surfaces. Under suitable conditions, we establish the short-time existence and uniqueness of the flow. We further introduce an extended version of the flow and prove its long-time existence. As an application, we prove the convergence result of the Yamabe flow in the case of hexagonal triangulations of the plane.

\end{abstract}


\section{Introduction}
\numberwithin{equation}{section}

The well-known Yamabe problem asks whether there exists a constant scalar curvature metric that is conformal to any given Riemannian metric \cite{MR125546}.
The answer is affirmative on closed Riemannian manifolds \cite{MR788292}. 
As an analog for discrete cases, Luo \cite{MR2100762} first introduced the notion of discrete conformal factors on the triangulated surface $M$ with a PL(piecewise linear) metric.
Similarly, fix a PL metric $d$ on $M$. The discrete Yamabe problem then asks whether there is a constant curvature PL metric within the
discrete conformal class $[d]$. Luo proposed to find the constant curvature PL metric by the combinatorial Yamabe flow. 
Luo considered the combinatorial Yamabe flow as a negative gradient flow of a locally convex functional $F$ and proved that the solution of the normalized combinatorial Yamabe flow converges exponentially fast to a constant curvature PL metric if no singularity develops. If some triangle breaks the triangle inequalities and degenerates to a line segment in finite time, we say that the flow develops removable singularities. In this case, the flow stops in finite time. In order to overcome this issue, Luo introduced a surgery on the triangulation called "edge-flipping". Whenever some triangle degenerates, say, the vertex $v_i$ is moving towards the interior of the edge $v_jv_k$ and $\Delta v_iv_jv_k$ and $\Delta v_lv_jv_k$ are two adjacent triangles. Then one can delete the edge $v_jv_k$ and add the new edge $v_iv_l$ and the flow continues after this operation. In his graduate thesis, Wu proved that if the flow converges, edge-flipping will occur only finitely many times. For detailed proofs, we recommend that the reader consult the English edition \cite{wu_flip}.

Ge-Jiang \cite{MR3566936} proposed an alternative method to address the problem of removable singularities. In their work, they studied the normalized combinatorial Yamabe flow on finitely triangulated surfaces. The core of their approach involved developing an extension of the flow that guarantees the existence of solutions for all time $t > 0$ without performing surgeries on the triangulations.

In the past few decades, there have been various important works on discrete conformal geometry. Bobenko-Pinkall-Springborn \cite{MR3375525} established a connection between discrete conformal geometry and the geometry of ideal hyperbolic polyhedra. They also obtained an explicit formula of the convex functional $F$ introduced in \cite{MR2100762}. Gu-Luo-Sun-Wu \cite{MR3807319} established a discrete uniformization theorem, which shows that each polyhedral metric
on a compact surface is discrete conformal to a constant curvature metric which is unique up to scaling. Other important related works can be referred to \cite{MR3825607, MR4466644, MR2399657, MR2788656} and others.

Most of the prior works mentioned above have focused on discrete geometry on finite triangulations of compact manifolds. It is also an important topic to study discrete conformal structures on infinite triangulations of noncompact surfaces. In the circle packing setting, Rodin and Sullivan \cite{MR906396} first obtained a rigidity result about hexagonal triangulations. He-Schramm \cite{MR1331923} and He \cite{MR1680531} yielded classification results that characterize circle packings on infinite triangulations of the complex plane. In the PL discrete conformal setting, Wu-Gu-Sun \cite{MR3356946} and Dai-Ge-Ma \cite{MR4389487} studied the rigidity problem of the infinite hexagonal triangulation of the complex plane. Dai-Wu \cite{MR4792293} extended the rigidity results to Delaunay triangulations of the plane. Other important related works can be referred to \cite{2404.11258, MR3457760} and others.

Recently, Ge-Hua-Zhou \cite{2504.05817} employed parabolic methods of partial differential equations to study the combinatorial Ricci flow introduced by Chow-Luo \cite{MR2015261} on infinite disk triangulations in Euclidean or hyperbolic background geometry.
Inspired by their work, we study the combinatorial Yamabe flow on infinitely triangulated surfaces. We prove the existence and uniqueness results under some additional assumptions. We also prove the long-time existence of the extended flow on infinite triangulations as an analog of Ge-Jiang's conclusion \cite{MR3566936}. After that, we prove the convergence result on hexagonal triangulations as an application.

Let $M$ be a surface, $T=(V,E,F)$ be a locally finite triangulation without boundary, where $V,E,F$ denote the set of vertices, edges and faces respectively. A PL (piecewise linear) metric on $T$ is a function $l : E \to \mathbb{R}_+ $ such that for every triangle $\Delta ijk \in  F$, the edge lengths $l_{ij},l_{ik}$ and $l_{jk}$ satisfy the triangle inequality.

\begin{definition}[\cite{MR2100762}]
    Suppose $l, \widetilde{l}$ be two PL metrics on $T=(V,E,F)$ and $u : V \to \mathbb{R} $ be a function. $l, \widetilde{l}$
    are discrete conformal with discrete conformal factor $u$, 
    denoted by $\widetilde{l} = u \ast l$, if for any edge $ij \in E$, 
    \begin{equation}
        \widetilde{l}_{ij} = e^{\frac{1}{2} (u_i + u_j)} l_{ij}.
    \end{equation}
\end{definition}

\begin{definition}[\cite{MR2100762}]
     Given a PL metric $d$ on a locally finite triangulation $T=(V,E,F)$, let $\theta^{jk}_i$ denote the inner angle at the vertex $i$ of the triangle $\Delta ijk$. The \textbf{dicrete Gaussian curvature}
    at each vertex $i$ is defined as
    \begin{equation}
        K_i = 2\pi - \sum_{\Delta ijk \in F} \theta^{jk}_i .
    \end{equation}
    The \textbf{combinatorial Yamabe flow} is defined as follows:
    \begin{align}
        \left\{\begin{aligned}
            \frac{\mathrm{d}u_i}{\mathrm{d}t} = -K_i \\
            u_i(0) = 0. 
        \end{aligned} \right.
        \label{CYF}
    \end{align}
    Here $K_i$ is the curvature of the PL metric $u \ast d$ at time $t$.
\end{definition}

We introduce some assumptions about the initial PL metric. 
\begin{definition}
    Given a PL metric $l$ on $T=(V,E,F)$, let $\theta^{jk}_i$ denote the inner angle at the vertex $i$ in the Euclidean triangle $\Delta ijk \in  F$.

    (a) $l$ is \textbf{uniformly nondegenerate} if there exists a constant $\epsilon > 0$ such that $\theta^{jk}_i\geqslant \epsilon$
    for any $\Delta ijk \in  F$.

    (b) $l$ is \textbf{uniformly Delaunay} if there exists a constant $\epsilon > 0$ such that $\theta^{jk_1}_i + \theta^{jk_2}_i \leqslant \pi - \epsilon$
    for any adjacent triangles $\Delta ijk_1, \Delta ijk_2 \in  F$.
\end{definition}

Now we present our main results.
\begin{theorem}[\textbf{Existence,I}] 
    \label{thm:existence}
    Let $T=(V,E,F)$ be an infinite triangulation with vertex degree bounded by $M$ and let $d$ be a PL metric on $T$ which is uniformly nondegenerate and uniformly Delaunay. Then there exists a positive constant $T_0 = T_0(\epsilon, M)$
    such that the flow (\ref{CYF}) has a solution $u \in C_t^{\infty}(V\times [0, T_0))$.
\end{theorem}

Next, we prove the long-time existence of the extended flow. We extend the definition of the discrete curvature $K$ continuously to $\widetilde{K}$, which is well-defined even when the triangle inequalities are not satisfied. Then the flow (\ref{CYF}) can also be naturally extended to 
\begin{align}
        \left\{\begin{aligned}
            \frac{\mathrm{d}u_i}{\mathrm{d}t} = -\widetilde{K}_i \\
            u_i(0) = 0 .
        \end{aligned} \right.
        \label{extended CYF}
\end{align}
The extended flow continues even if the triangle inequalities fail in some finite time. Thus, we can prove that the solution of the extended flow (\ref{extended CYF}) exists for all time $t \geqslant 0$.

\begin{theorem}[\textbf{Existence,II}]
    \label{thm:existence2}
    Let $T=(V,E,F)$ be an infinite triangulation and $d$ be a fixed initial PL metric. Then there exists a global solution $u \in C_t^{1}(V\times [0, \infty))$ for all time $t \geqslant 0$ to the extended flow (\ref{extended CYF}). 
\end{theorem}

\begin{theorem}[\textbf{Uniqueness}] 
    \label{thm:uniqueness}
    Let $T=(V,E,F)$ be an infinite triangulation with vertex degree bouneded by $M$ and let $d$ be a PL metric on $T$ that is uniformly nondegenerate and uniformly Delaunay.
    Let $u(t), \hat {u}(t)$ be two solutions of the flow (\ref{CYF})
    for $t \in [0, T]$. If $(su + (1-s)\hat {u}) \ast d$ are uniformly Delaunay for $t \in [0, T]$ and $s \in [0,1]$, then $u\equiv \hat {u}$.
\end{theorem}

As an application, we consider the infinite hexagonal triangulation of the plane where every vertex is of degree 6, denoted by $T_H = (V_H, E_H, F_H)$. The regular metric on $T_H$, denoted by $d_{reg}$, is the constant metric, which yields the tiling of the plane by regular equilateral triangles. Let $\phi: V_H \to \mathbb{R} $ be a conformal factor on $T_H$. We consider the combinatorial Yamabe flow with initial metric $\phi \ast d_{reg}$.
\begin{align}
    \left\{ \begin{aligned}
         \frac{\mathrm{d}u_i}{\mathrm{d}t} = -K_i \\
          u_i(0) = \phi .
    \end{aligned} \right.
    \label{CYF2}
\end{align}
We prove the convergence of the combinatorial Yamabe flow provided that the initial metric is a small perturbation of the regular one.

\begin{theorem}[\textbf{Convergence on hexagonal triangulations}]
    \label{Convergence}
    Let $T_H = (V_H, E_H, F_H)$ be the infinite hexagonal triangulation of the plane and let $d = d_{reg} \equiv 1$ be the regular metric on $T_H$. There is a universal constant $\epsilon_0 > 0$ such that for any conformal factor $\phi:V_H \to \mathbb{R}$ with $\Vert \phi \Vert_{l^2} \leqslant \epsilon_0$, there exists a unique global solution $u(t)$ to the flow (\ref{CYF2}). Moreover, $u(t)$ converges to the regular metric.
\end{theorem}

\section{Preliminaries}

\subsection{Variational principles}

Now we introduce the variational principle for the vertex scaling in Euclidean background geometry.

\begin{lemma}[\cite{MR2100762}]
    \label{va1}
    Let $\Delta v_1v_2v_3$ be a Euclidean triangle. Let $u_1, u_2, u_3$ be discrete conformal factors and $\theta_1^{23}, \theta_2^{13}, \theta_3^{12}$ be inner angles of $\Delta v_1v_2v_3$ at the vertices $v_1, v_2, v_3$ respectively. Then for $\{i,j,k\} = \{1,2,3\}$, we have
    \begin{equation}
        \frac{\partial \theta_i^{jk}}{\partial u_j} = \frac{\partial \theta_j^{ik}}{\partial u_i} = \frac{1}{2} \cot \theta_k^{ij}
    \end{equation}
    \begin{equation}
        \frac{\partial \theta_i^{jk}}{\partial u_i} = -\frac{1}{2} (\cot \theta_j^{ik} + \cot \theta_k^{ij})
    \end{equation}
\end{lemma}

The proof of Lemma \ref{va1} follows from a direct computation.

Given an undirected weighted graph $G=(V, E, \omega)$, where $V$ and $E$ are the sets of vertices and edges, respectively, and $\omega:E \to [0, \infty)$ is the edge weight function. We write $i \thicksim j$ if the vertex $v_i$ is adjacent to $v_j$. For a vertex function $f:V \to \mathbb{R} $, the discrete Laplacian $\Delta_{\omega}$ associated with the edge weight is given by
\begin{equation}
    \Delta_{\omega}f_i = \sum_{j:j \thicksim i}\omega_{ij}(f_j-f_i), \forall i \in V.
\end{equation}

Next, we introduce the curvature evolution equation under the combinatorial Yamabe flow.
\begin{proposition}[\cite{MR2100762}]
     Fix a PL metric $d$ on a locally finite 
    triangulation $T=(V,E,F)$ and let $u: V \to \mathbb{R}$ be a discrete 
    conformal factor. Then
    \begin{equation}
        \mathrm{d}K_i = - \sum_{ij \in E} \mu_{ij} (\mathrm{d}u_j - \mathrm{d}u_i)
    \end{equation}
    and here
    \begin{equation}
        \mu_{ij} = \frac{1}{2}(\cot \theta^{ij}_{k_1}+ \cot \theta^{ij}_{k_2}) \label{weight}
    \end{equation}
    for adjacent triangles $\Delta ijk_1, \Delta ijk_2 \in  F$.
    And under combinatorial Yamabe flow, one has
    \begin{align}
        \frac{\mathrm{d}K_i}{\mathrm{d}t} &= \sum_{ij \in E} \mu_{ij} (K_j - K_i) \nonumber \\
                        &= \Delta_{\mu(t)}K_i
        \label{cur evolution}
    \end{align}
    where $\Delta_{\mu(t)}$ is the discrete Laplacian operator defined on 
    the weighted graph $G = (V,E,\mu(t))$.
\end{proposition}

\subsection{Extension of the flow}
In this section, we extend the definition of inner angles so that it is still well-defined even when the triangle inequalities are not satisfied. Set
\begin{equation*}
    \Omega = \{(l_1, l_2,l_3) \in \mathbb{R}_{>0}^3 | l_1 + l_2 > l_3, l_1 + l_3 > l_2, l_2 + l_3 > l_1\}.
\end{equation*}
It is clear that there exists an Euclidean triangle with edge lengths $l_1, l_2, l_3$ for $l = (l_1, l_2, l_3) \in \Omega$, and we denote by $\theta_1, \theta_2$ and $\theta_3$ the corresponding inner angles.  We call the map $\theta:\Omega \to (0, \pi)^3$ that maps $(l_1, l_2, l_3)$ to $(\theta_1, \theta_2, \theta_3)$  the inner angle map. Next, we extend $\theta$ continuously to $\widetilde{\theta}: \mathbb{R}_{>0}^3 \to [0, \pi]^3$ as follows.

\begin{definition}[\cite{MR2862158}]
    \label{def extension1}
    Let $l = (l_1, l_2,l_3) \in \mathbb{R}_{>0}^3$ be the lengths of the edges. If $l \in \Omega$, then define $\widetilde{\theta}(l_1, l_2, l_3) = \theta(l_1, l_2, l_3)$ and $l$ forms a nondegenerate Euclidean triangle with inner angles $\widetilde{\theta}$. If $l$ does not satisfy the triangle inequalities, say $l_i \geqslant l_j + l_k$ for example, then define $\widetilde{\theta}_i = \pi$, $\widetilde{\theta}_j =\widetilde{\theta}_k = 0$ where $\{i,j,k\} = \{1,2,3\}$. If the latter condition holds, we call such a triangle a generalized triangle.
\end{definition}
Thus, the inner angle map $\theta:\Omega \to (0, \pi)^3$ is extended to $\widetilde{\theta}: \mathbb{R}_{>0}^3 \to [0, \pi]^3$. It remains to show that $\widetilde{\theta}$ is continuous in $\mathbb{R}_{>0}^3$.

\begin{lemma}
    The extended inner angle map $\widetilde{\theta}$ is continuous in $\mathbb{R}_{>0}^3$ with $\widetilde{\theta}_1 + \widetilde{\theta}_2 + \widetilde{\theta}_3 = \pi$.
\end{lemma}
\begin{proof}
    It is only necessary to show that $\widetilde{\theta}$ is continuous on the boundary of $\Omega$. Let $L = (L_1,L_2,L_3) \in \overline{\Omega} - \Omega $ be a boundary point. Without loss of generality, we assume that $L_1 = L_2 + L_3$. Then the continuity
    \begin{equation*}
        \lim_{l \to L} \widetilde{\theta}_1 = \pi \; and \;  \lim_{l \to L} \widetilde{\theta}_2 = \lim_{l \to L} \widetilde{\theta}_3 = 0 
    \end{equation*} 
    follows from the cosine laws.
\end{proof}

Given a PL metric $d$ on a locally finite triangulation $T = (V,E,F)$ and a discrete conformal factor $u:V \to \mathbb{R}$, we call $u \ast d$ a pseudo PL metric if there exists a generalized triangle under the metric $u \ast d$ that does not satisfy the triangle inequalities. Naturally, we extend the definition of discrete curvature on a pseudo PL metric
\begin{equation}
    \widetilde{K}_i = 2\pi - \sum_{\Delta ijk \in F} \widetilde{\theta}^{jk}_i.
\end{equation} 
The extended combinatorial Yamabe flow is also naturally defined:
\begin{align}
        \left\{\begin{aligned}
            \frac{\mathrm{d}u_i}{\mathrm{d}t} = -\widetilde{K}_i \\
            u_i(0) = 0 .
        \end{aligned} \right.
        \label{extended CYF2}
\end{align}
Here $\widetilde{K}_i$ is the extended curvature of the PL metric $u \ast d$ at time $t$.

\section{Proof of main results}
First, we introduce a lemma from elementary geometry that helps us to compute the existence time of the flow.

\begin{lemma}
    \label{lem:triangle}
    For any Euclidean triangle $\Delta ijk$, we denote by $\theta_i, \theta_j$ and $\theta_k$ the inner angles of the triangle, and by $d_{ij}, d_{ik}$ and $ d_{jk}$ the edge lengths. Assume that $\theta_i, \theta_j, \theta_k \geqslant \epsilon > 0$. Then there exists a constant $\delta = \delta(\epsilon)$ such that for every function $u:V \to \mathbb{R} $ satisfying 
    $\Vert u \Vert_{L^\infty} \leqslant \delta $, the inner angles $\widetilde{\theta_i}, \widetilde{\theta_j}$ and $\widetilde{\theta_k}$ under the conformal metric $\widetilde{d} = u \ast d$ satisfy $|\widetilde{\theta_r} - \theta_r| \leqslant \frac{\epsilon}{2}$ for $r \in \{i,j,k\}$.  
\end{lemma}
Since the proof is elementary, we leave the proof to the Appendix \ref{appendix}.

Let $T=(V,E,F)$ be a locally finite triangulation on a surface $M$ and $d$
be a PL metric on $T$. Pick a sequence of finite subsets $V_i$ of $V$ 
such that
$$
V_i \subset V_{i+1},\: \bigcup_{i=1}^{\infty} V_i = V.
$$
Consider the Yamabe flow on $V_i$,
\begin{align}
        \left\{\begin{aligned}
            &\frac{\mathrm{d}u_j^{[i]}}{\mathrm{d}t} = -K_j,   \;  \forall j \in Int(V_i)\\
            &u_j^{[i]}(0) = 0, \; \forall (j,t) \in (V_i \times \{0\}) \cup (\partial V_i \times (0, \infty)).
        \end{aligned} \right.
\end{align}
The short-time existence of the above equations follows from basic ODE theory since $V_i$ is finite.

\begin{proof}[\textbf{Proof of Theorem \ref{thm:existence}}]
    By the above lemma, there exists a constant $\delta_0 = \delta_0(\epsilon)$
    such that for every function $u:V \to \mathbb{R} $ satisfying 
    $\Vert u \Vert_{L^\infty}<\delta_0 $, $u \ast d$ satisfies triangle
    inequalities and for any $\Delta ijk \in  F$,
    \begin{equation}\label{3.2}
        \theta^{jk}_i(u \ast d)\geqslant \frac{\epsilon}{2},
    \end{equation}
    for any adjacent triangles $\Delta ijk_1, \Delta ijk_2 \in  F$, we have
    \begin{equation}\label{3.3}
        \theta^{jk_1}_i(u \ast d) + \theta^{jk_2}_i(u \ast d) \leqslant \pi - \frac{\epsilon}{2}.
    \end{equation}
    
    Fix a vertex $v_j \in V$, choose a sufficiently large $i$ such that $v_j \in Int(V_i)$. By the definition of $K_i$, one has
    \begin{equation}\label{3.4}
        |K_j(u)| \leqslant (2+deg(v_j))\pi \leqslant (2+M)\pi. 
    \end{equation}
    Therefore, there exists a positive constant $T_0 = T_0(\epsilon, M)$ 
    such that in the time interval $t \in [0, T_0)$, 
    \begin{equation*}
        |u_j^{[i]}(t)| \leqslant (2 + M)\pi T_0 < \delta_0.
    \end{equation*}
Therefore, estimates \eqref{3.2} and \eqref{3.3} hold for $u(t)*d, \forall t\in[0,T_0).$ 
    By the curvature evolution formula (\ref{cur evolution}),
    \begin{equation}\label{3.6}
        \mu_{ij}(t) \leqslant \cot \frac{\epsilon}{2}, \; \forall t \in [0, T_0).
    \end{equation}
    Then by (\ref{3.4}) and (\ref{3.6}), we have
    \begin{equation*}
        \sup_i \Vert u_j^{[i]}(t)\Vert_{C^2[0, T_0)} \leqslant C(\epsilon, M).
    \end{equation*}
    Then we can apply the Arzela-Ascoli theorem and the standard diagonal argument. There is a subsequence of $u^{[i]}(t)$
    (still denoted by $u^{[i]}(t)$) such that $\forall j \in V, \; u_j^{[i]}(t) \to u_j^{\ast}(t)$
    in $C^1[0, T_0)$ for some $u^{\ast}: V \to \mathbb{R}$, which is the solution of (\ref{CYF}) in the time
    interval $[0, T_0)$. Moreover, the smoothness of $u^{\ast}(t)$ with respect to $t$ follows from the smoothness of $K(u)$.

\end{proof}

For extended flow (\ref{extended CYF}) on a finite triangulation, the long-time existence is ensured by Peano's
existence theorem in the classical ODE theory \cite{MR3566936}. We first establish the long-time existence of the solution to the extended flow (\ref{extended CYF}) on an infinite triangulation. To see the result, we first introduce the following lemma.

\begin{lemma}
    \label{lem:lip}
    Let $f:\mathbb{R}^n \to \mathbb{R}$ be a Lipschitz function. Assume that $g$ is a continuous function and $\nabla f = g$ in the weak sense, then $f \in C^1$ and $\nabla f = g$ everywhere.
\end{lemma}
The proof of this lemma can be found in \cite[Section 5.8]{MR1625845}. 

\begin{proof}[\textbf{Proof of Theorem \ref{thm:existence2}}]
    Consider the extended Yamabe flow on $V_i$
    \begin{align}\label{3.8}
        \left\{\begin{aligned}
            &\frac{\mathrm{d}u_j^{[i]}}{\mathrm{d}t} = -\widetilde{K}_j,   \;  \forall j \in Int(V_i)\\
            &u_j^{[i]}(0) = 0, \; \forall (j,t) \in (V_i \times \{0\}) \cup (\partial V_i \times (0, \infty)).
        \end{aligned} \right.
    \end{align}
    Fix a vertex $v_j \in V$, choose a sufficiently large $i$ such that $v_j \in Int(V_i)$. By the definition of $\widetilde{K}_i$, one has
    \begin{equation}\label{3.9}
        |\widetilde{K}_j(u)| \leqslant (2+deg(v_j))\pi .
    \end{equation}
    Fix the time interval $[0, T]$ for any $T>0$. By \eqref{3.8} and \eqref{3.9}
    \begin{equation*}
        \sup_i \Vert u_j^{[i]}(t)\Vert_{C^1[0, T)} \leqslant C(deg(v_j), T).
    \end{equation*}
    Then by the Arzela-Ascoli theorm and the standard diagonal argument, there exists a subsequence of $u^{[i]}(t)$(still denoted by $u^{[i]}(t)$) such that $\forall j \in V, \; u_j^{[i]}(t) \to u_j^{\ast}(t)$ uniformly for some $u^{\ast}: V \to \mathbb{R}$. It remains to show that $u^{\ast} \in C_t^{1}(V\times [0, T])$.

    Since $u_j^{[i]}$ are Lipschitz continuous with the Lipschitz constant bounded by $(2+deg(v_j))\pi$, then $u^{\ast}_j$ is also Lipschitz continuous after taking the limit. Let $\phi$ be a smooth test function with $\phi(0) = \phi(T) = 0$. Then by the dominated convergence theorem, 
    \begin{align*}
        \int_0^T u_j^{\ast} \phi' \mathrm{d}t = \lim _{i \to \infty}  \int_0^T u_j^{[i]} \phi' \mathrm{d}t = \lim _{i \to \infty}  \int_0^T \widetilde{K}_j(u^{[i]}) \phi \mathrm{d}t =  \int_0^T \widetilde{K}_j(u^{\ast}) \phi \mathrm{d}t.
    \end{align*} 
    The last equality follows from the continuity of the extended curvature $\widetilde{K}$. Thus, $\frac{\mathrm{d}}{\mathrm{d}t}u_j^{\ast} = -\widetilde{K}_j(u^{\ast})$ in the weak sense. By Lemma \ref{lem:lip}, we have $u_j^{\ast} \in C^1[0,T]$.
\end{proof}

\begin{remark}
    Let $u(t)$ be a solution of the extended flow (\ref{extended CYF}). By the variational principle (\ref{weight}), it is clear that $|\frac{\partial K_i}{\partial u_j}| = \infty$ when some triangle degenerates. Thus, $\widetilde{K}_i$ is generally not Lipschitz continuous, and the uniqueness of the solution to the extended flow remains unknown. 
\end{remark}

To prove the uniqueness of the flow, we first introduce a maximal principle of nonlinear heat equations in the discrete setting, which is first obtained in Ge-Hua-Zhou \cite{2504.05817}.
\begin{lemma}
    \label{lem:max principle}
    Let $G=(V,E)$ be an undirected infinite graph and let $\{\omega(t)\}_{t \geqslant 0}$
    be a one-parameter family of nonnegative weights on $E$ for $t \in [0,T]$
    such that
    \begin{equation}
        \sum_{j:j\thicksim i}\omega_{ij}(t) \leqslant C, \; \forall (j,t) \in V \times [0,T]. 
    \end{equation}
    Suppose $f:V \times [0, T] \to \mathbb{R}$ is a bounded solution to 
    \begin{equation}
        \frac{\mathrm{d}f(t)}{\mathrm{d}t} = \Delta_{\omega(t)}f(t)
    \end{equation}
    for $t \in [0, T]$ with $f(0) = 0$ on $V$, then $f(t) \equiv 0$ 
    for $t \in [0, T]$.
\end{lemma}
For the completeness of the proof, we prove the maximum principle as follows.
 \begin{proof}
     It is only necessary to prove that if $f:V \times [0, T] \to \mathbb{R}$ is a bounded solution to 
     \begin{equation*}
         \frac{\mathrm{d}f(t)}{\mathrm{d}t} \leqslant \Delta_{\omega(t)}f(t)
     \end{equation*}
     with $f(0) \leqslant 0$ on $V$, then $f(t) \leqslant 0$ 
     for $t \in [0, T]$. Fixing a vertex $v_0 \in V$, we denote by $d_0(i) = dist(v_0, v_i)$ the distance between $v_0$ and $v_i$. Then we have
     \begin{equation*}
         \Delta_{\omega(t)}d_0(i) \leqslant \sum_{j:j\thicksim i}\omega_{ij} \leqslant C
     \end{equation*}
     For $\delta > 0$, we set $f_{\delta} = f - \delta d_0 - (C+1)\delta t$, then
     \begin{align}
         \frac{\mathrm{d}f_{\delta}}{\mathrm{d}t} &= \Delta_{\omega(t)}f- (C+1)\delta \nonumber\\
         &= \Delta_{\omega(t)}f_{\delta} +(\Delta_{\omega(t)}d_0 - C - 1)\delta \nonumber\\
         &< \Delta_{\omega(t)}f_{\delta}.
         \label{3.16}
     \end{align}
     Since $f$ is bounded, then $f_{\delta}(v_i, t) \to -\infty$ for any $t \in [0,T]$ as $d_0(i) \to \infty$. Therefore, we assume that $f_{\delta}$ attains its maximum at $(v_{max}, t_{max})$. If $t_{max} = 0$, then $f_{\delta}(v,t) \leqslant f_{\delta}(v_{max}, 0) \leqslant f(v_{max}, 0) \leqslant 0$ for any $(v,t) \in V \times [0,T]$. If $t_{max} \in (0,T]$, then $\frac{df_{\delta}}{dt}(v_{max}, t_{max}) \geqslant 0$ and $\Delta_{\omega(t)}f_{\delta}(v_{max}, t_{max}) \leqslant 0$, which leads to a contradiction to (\ref{3.16}). Hence, we have $f_{\delta} \leqslant 0$ on $V \times [0,T]$. Since $\delta$ is arbitrary, we have $f \leq  0$ by letting $\delta \to 0$.

 \end{proof}

\begin{proof}[\textbf{Proof of Theorem \ref{thm:uniqueness}}]
    Let $u(t), \hat {u}(t)$ be two solutions of the flow (\ref{CYF}) with 
    $\mu(t), \hat {\mu}(t)$ defined in (\ref{weight}).
    By the variational principles and the Newton-Leibniz formula, we have
    \begin{align*}
        &\frac{\mathrm{d}(u(t)-\hat {u}(t))}{\mathrm{d}t} = - (K_{i}(u(t))-K_{i}(\hat{u}(t))) \nonumber\\
        &=-\sum_{j: j \sim i} \int_{0}^{1} \frac{\partial K_{i}}{\partial u_{j}}(s u(t)+(1-s) \hat{u}(t)) \mathrm{d} s \cdot\left[u_{j}(t)-\hat{u}_{j}(t)\right]- \nonumber\\
        &\; \; \; \; \int_{0}^{1} \frac{\partial K_{i}}{\partial u_{i}}(s u(t)+(1-s) \hat{u}(t)) \mathrm{d} s \cdot\left[u_{i}(t)-\hat{u}_{i}(t)\right] \nonumber\\
        &= \sum_{j: j \sim i} \int_{0}^{1} \mu_{ij}(su(t) + (1-s)\hat {u}(t))\mathrm{d}s \cdot\left[(u_{j}(t)-\hat{u}_{j}(t)) - (u_{i}(t)-\hat{u}_{i}(t))\right].
    \end{align*}
    We write the above equation as  
    \begin{equation*}
        \frac{\mathrm{d}(u(t)-\hat {u}(t))}{\mathrm{d}t} = \Delta_{\omega(t)}(u(t)-\hat {u}(t)),
    \end{equation*}
    where
    \begin{equation*}
        \omega_{ij}(t) = \int_{0}^{1} \mu_{ij}(su(t) + (1-s)\hat {u}(t)) \,\mathrm{d}s. 
    \end{equation*}
    It's straightforward to check that if $l(su(t) + (1-s)\hat {u}(t))$ is 
    uniformly Delaunay for $s \in [0,1]$, then $\omega_{ij}(t) > \epsilon$ 
    for some constant $\epsilon > 0$.

    Write $l_{ij} = l_{ij}(u),\; \hat{l}_{ij} =  l_{ij}(\hat{u})$ for simplicity. By the flow (\ref{CYF}) we have 
    \begin{equation}
        -(2+M)\pi T \leqslant u_i \leqslant (2+M)\pi T, \;  \forall i \in V.
        \label{3.13}
    \end{equation}
    Thus, 
    \begin{equation*}
        d_{ij}e^{-(2+M)\pi T} \leqslant l_{ij}, \hat{l}_{ij} \leqslant d_{ij}e^{(2+M)\pi T}, \;  \forall ij \in E.
    \end{equation*}
    For $s \in [0,1]$,
    \begin{align*}
        l_{ij}(su + (1-s) \hat{u}) &= e^{\frac{1}{2}(s(u_i + u_j) + (1-s)(\hat{u}_i + \hat{u}_j))}d_{ij}\nonumber\\
        &= (e^{\frac{1}{2}(u_i + u_j)}d_{ij})^s \cdot (e^{\frac{1}{2}(\hat{u}_i + \hat{u}_j)}d_{ij})^{1-s}\nonumber\\      
        &= l_{ij}^s \cdot \hat{l}_{ij}^{1-s}.               
    \end{align*}
    Thus, 
    \begin{equation*}
        d_{ij}e^{-(2+M)\pi T} \leqslant l_{ij}(su + (1-s) \hat{u}) \leqslant d_{ij}e^{(2+M)\pi T}, \;  \forall ij \in E.
    \end{equation*}
    Next, we claim that all inner angles $\theta^{jk}_i(su + (1-s) \hat{u})$
    lies in a compact subset of $(0, \pi)$. Fix $\Delta ijk \in F$ under 
    the metric $(su + (1-s) \hat{u}) \ast d$. Since $\theta^{jk}_i \leqslant \pi - \epsilon$,
    either $\theta^{ik}_j$ or $\theta^{ij}_k$ must be greater than $\frac{\epsilon}{2}$. Without loss of generality, assume that $\theta^{ik}_j>\frac{\epsilon}{2}$. Note the initial metric $d$ is uniformly nondegenerate, then by the sine law,
    \begin{equation*}
        \sin \theta^{jk}_i = \frac{l_{jk}}{l_{ik}}\sin \theta^{ik}_j \geqslant \frac{d_{jk}}{d_{ik}}e^{-(4+2M)T}\sin \frac{\epsilon}{2} \geqslant e^{-(4+2M)T}\sin{\epsilon}\sin \frac{\epsilon}{2}.
    \end{equation*}
    Thus 
    \begin{equation*}
        0 < C_1(\epsilon, M, T) \leqslant \theta^{jk}_i \leqslant \pi - \epsilon.
    \end{equation*}
    Then, by the boundedness of inner angles,
    \begin{align*}
        \mu_{ij}(su(t) + (1-s)\hat {u}(t)) \leqslant C_2(\epsilon, M, T),\\
        \omega_{ij}(t) \leqslant C_3(\epsilon, M, T).
    \end{align*}
    Since $T$ has bounded degree,
    \begin{equation*}
        \sum_{j:j\thicksim i}\omega_{ij}(t) \leqslant C(\epsilon, M, T).
    \end{equation*}
    Let $g(t) = u(t) - \hat{u}(t)$, we have
    \begin{equation*}
        \frac{\mathrm{d}g(t)}{\mathrm{d}t} = \Delta_{\omega(t)}g(t).
    \end{equation*}
    Moreover, $|g(t)| \leqslant (4+2M)\pi T$ by (\ref{3.13}). Then $u(t) \equiv \hat{u}(t)$
    follows from Lemma \ref{lem:max principle}.
\end{proof}

Next, we consider the infinite hexagonal triangulation of the plane $T_H = (V_H, E_H, F_H)$ equipped with a constant metric $d = d_{reg} \equiv 1$. Let $\phi: V_H \to \mathbb{R} $ be a conformal factor on $T_H$. We consider the combinatorial Yamabe flow \ref{CYF2} with initial metric $\phi \ast d$, which is a small perturbation of the regular metric. We apply a method first introduced in Ge-Hua-Zhou \cite{2504.05817} to reformulate (\ref{CYF2}) as a semilinear parabolic equation. Let $u:V_H \to \mathbb{R}$ be a conformal factor on $T_H$ and $l = u \ast d$ be conformal to the regular metric. Then for $\Delta ijk \in F_H$,
\begin{align*}
    \theta_i^{jk} &= \arccos \frac{l_{ij}^2 + l_{ik}^2 - l_{jk}^2}{2l_{ij}l_{ik}}\nonumber\\
    &= \arccos \frac{e^{u_i + u_j} + e^{u_i + u_k} - e^{u_j + u_k}}{2e^{u_i + \frac{1}{2}(u_j + u_k)}}\nonumber\\
    &= G(u_j - u_i, u_k - u_i),
\end{align*}
where 
\begin{equation*}
    G(x, y) = \arccos \frac{e^x + e^y -e^{x + y}}{2e^{\frac{1}{2}(x + y)}}.
\end{equation*}
Since $G(0,0) = \frac{\pi}{3}$, $G_x(0,0) = G_y(0,0) = \frac{\sqrt{3}}{6}$, we write the Taylor expansion of $G(x,y)$ as
\begin{equation*}
    G(x,y) = \frac{\pi}{3} + \frac{\sqrt{3}}{6}(x + y) + \widetilde{F}(x, y),
\end{equation*}
where $\widetilde{F}(x, y)$ satisfies $\widetilde{F}(z) = O(|z|^2)$.
Then for any vertex $i \in V_F$, we have
\begin{align}
    \frac{\mathrm{d}u_i}{\mathrm{d}t} &= -2\pi + \sum_{ijk \in F_H}G(u_j - u_i, u_k - u_i)\nonumber\\
    &= \sum_{ijk \in F_H} \frac{\sqrt{3}}{6}(u_j - u_i) + \frac{\sqrt{3}}{6}(u_k - u_i) + \widetilde{F}(u_j - u_i, u_k - u_i)\nonumber\\
    &= \Delta_c u_i + F(Du)(i),
    \label{semilinear eq}
\end{align}
where $\Delta_c$ is the discrete Laplacian operator with constant weight $c = \frac{\sqrt{3}}{3}$ and 
\begin{equation*}
    F(Du)(i) = \sum_{ijk \in F_H}\widetilde{F}(u_j - u_i, u_k - u_i).
\end{equation*}

Following an approach similar to Theorem 5.4 in Ge-Hua-Zhou\cite{2504.05817}, we have the following lemma.
\begin{lemma}
    [\cite{2504.05817}]
    \label{convergence}
    There exist constants $\epsilon_0 > 0$ and $T_0 = T_0(\epsilon_0) > 0$ such that for any initial value $\Vert \phi \Vert_{l^2} < \epsilon_0$, the semilinear parabolic equation
    \begin{align}
             \label{semilinear equation}
             \left\{ \begin{aligned}
             \frac{\mathrm{d}u_i}{\mathrm{d}t} -  \Delta_c u_i &= F(Du)(i)\\
              u(0) &= \phi
              \end{aligned} \right.
    \end{align}
    admits a solution $u(t)$ on $[0, T_0)$ such that 
    \begin{equation}
        \frac{\mathrm{d}}{\mathrm{d}t}\|u(t)\|_{l^2}^2+\frac{\sqrt{3}}{3}\mathcal{E}(u(t))\leq0,
        \label{energy}
    \end{equation}
    where the Dirichlet energy of a function $u \in \mathbb{R}^{V_H}$ is defined as
    \begin{equation*}
        \mathcal{E}(u) = \sum_{ij \in E_H} |u_i - u_j|^2.
    \end{equation*}

\end{lemma}

\begin{proof}[\textbf{Proof of Theorem \ref{Convergence}}]

    By Lemma \ref{lem:triangle}, we can choose sufficiently small $\epsilon_1$ such that for any $u:V_H \to \mathbb{R}$ with $\Vert u - \phi \Vert_{l^\infty} \leqslant 2\epsilon_1$, $u \ast d$ satisfies triangle inequalities and for any $\Delta ijk \in  F_H$,
    \begin{equation}
        \theta^{jk}_i(u \ast d)\geqslant \frac{\pi}{6},
        \label{nondegenerate}
    \end{equation}
    and for any adjacent triangles $\Delta ijk_1, \Delta 
    ijk_2 \in  F_H$,
    \begin{equation}
        \theta^{jk_1}_i(u \ast d) + \theta^{jk_2}_i(u \ast d) \leqslant \frac{5\pi}{6}.
        \label{Delaunay}
    \end{equation}

    By Lemma \ref{convergence}, there exist constants $\epsilon_2 > 0$ and $T_0 = T_0(\epsilon_2) > 0$ and we obtain the solution of (\ref{semilinear equation}) on $[0,T_0)$ provided that $\Vert \phi \Vert_{l^2} < \epsilon_2$. Set $\epsilon_0 = \min \{\epsilon_1, \epsilon_2\}$ and assume that $\Vert \phi \Vert_{l^2} < \epsilon_0$. Thus, the solution satisfies $\sup_{t \in [0, \delta)} \Vert u(t) \Vert_{l^2} \leqslant \epsilon_0$ and the metric $u(t) \ast d$ satisfies (\ref{nondegenerate}),(\ref{Delaunay}) in $[0, T_0)$. It is clear that if $u$ and $\hat{u}$ are two solutions in $[0, T_0)$ then $(su + (1-s)\hat {u}) \ast d$ are uniformly Delaunay for $t \in [0, T_0)$ and $s \in [0,1]$. Then the uniqueness follows from Theorem \ref{thm:uniqueness}.

    Noticing that the $l^2$ norm of $u(t)$ is non-increasing by Lemma \ref{convergence}, we can extend the solution to $[T_0, 2T_0]$, $[2T_0, 3T_0]$ and to $[0, \infty)$ finally such that the metric $u(t) \ast d$ satisfies (\ref{nondegenerate}),(\ref{Delaunay}) in $[0, \infty)$. Therefore, the flow (\ref{CYF2}) admits a unique global solution $u(t)$, with which (\ref{energy}) holds in $[0, \infty)$. Therefore, by (\ref{energy}) we have 
    \begin{equation*}
        \int_{0}^{\infty}\mathcal{E}(u(t))\mathrm{d}t < \infty.
    \end{equation*}
    Then $\mathcal{E}(u(t)) \to 0$ as $t \to \infty$ and $u(t) \ast d$ converges to the regular metric on $T_H$.

\end{proof}

\section{Appendix}\label{appendix}
We will prove Lemma \ref{lem:triangle} in this appendix.

\begin{proof}[\textbf{Proof of Lemma \ref{lem:triangle}}]
    Since the triangle $\Delta ijk$ is uniformly nondegenerate, by the sine law, 
    \begin{equation*}
        \sin \epsilon \leqslant \frac{d_{ij}}{d_{ik}}, \frac{d_{ij}}{d_{jk}}, \frac{d_{jk}}{d_{ik}} \leqslant \frac{1}{\sin \epsilon} 
    \end{equation*}
    Let $\delta = \delta(\epsilon)$ be a constant to be determined, then
    \begin{equation*}
        e^{-\delta}d_{ij} \leqslant \widetilde{d}_{ij} \leqslant e^{\delta}d_{ij}, \;  \forall ij \in E 
    \end{equation*}
    Thus,
    \begin{align*}
        |\cos \theta_i - \cos \widetilde{\theta}_i| &= \left | \frac{d_{ij}^2 + d_{ik}^2 - d_{jk}^2}{2d_{ij}d_{ik}}-\frac{\widetilde{d}_{ij}^2 + \widetilde{d}_{ik}^2 - \widetilde{d}_{jk}^2}{2\widetilde{d}_{ij}\widetilde{d}_{ik}}\right | \nonumber\\
        &\leqslant \left |\frac{d_{ij}}{2d_{ik}}-\frac{\widetilde{d}_{ij}}{2\widetilde{d}_{ik}} \right | + \left |\frac{d_{ik}}{2d_{ij}}-\frac{\widetilde{d}_{ik}}{2\widetilde{d}_{ij}}\right | + \left |\frac{d_{jk}^2}{2d_{ij}d_{ik}} - \frac{\widetilde{d}_{jk}^2}{2\widetilde{d}_{ij}\widetilde{d}_{ik}}\right |\nonumber\\
        &\leqslant \frac{e^{2\delta}-1}{\sin \epsilon} + \frac{e^{4\delta}-1}{2(\sin \epsilon)^2}\nonumber\\
        &\leqslant\frac{3(e^{4\delta}-1)}{2(\sin \epsilon)^2}.
    \end{align*}
    Let
    \begin{equation*}
        \delta = \frac{1}{4}\log (1 + \frac{2(\sin \epsilon)^2}{3}(1-\cos \frac{\epsilon}{4})),
    \end{equation*}
    we obtain $|\theta_i - \widetilde{\theta}_i| \leqslant \frac{\epsilon}{2}$ immediately. Moreover, it is clear to verify that $\delta \sim C\epsilon^4 $ when $\epsilon$ is sufficiently close to 0, where $C$ is a universal constant.

\end{proof}

\textbf{Acknowledgements.}
The author expresses sincere gratitude to Bobo Hua for his unwavering support and for proposing the research topics. The author also wishes to thank Puchun Zhou for his valuable suggestions, which significantly enhanced the clarity, coherence, and readability of the manuscript. In addition, the author is grateful to Xinrong Zhao and Wenhuang Chen for their insightful discussions throughout the course of this research.

\nocite{*}
\bibliographystyle{unsrt}
\bibliography{reference}

\noindent Bohao Ji, bhji24@m.fudan.edu.cn\\
\emph{School of Mathematical Sciences, Fudan University, Shanghai, 200433, P.R. China}\\
    
\end{document}